\documentclass[a4paper,12pt]{amsart}

\usepackage{amssymb}
\usepackage{graphicx}
\usepackage{float}
\usepackage{amsmath}
\usepackage{array}
\usepackage{multirow}
\usepackage{mathrsfs}
\usepackage{textcomp}
\usepackage[bottom]{footmisc}
\usepackage{xcolor}
\usepackage{cite}

\newcolumntype{M}[1]{>{\raggedright}m{#1}}

\DeclareMathAlphabet{\mathpzc}{OT1}{pzc}{m}{it}

\newtheorem{theorem}{Theorem}[section]
\newtheorem{lemma}[theorem]{Lemma}

\newtheorem{corollary}[theorem]{Corollary}
\newtheorem{conjecture}[theorem]{Conjecture}

\newtheorem{claim}[theorem]{Claim}

\theoremstyle{definition}
\newtheorem{definition}[theorem]{Definition}
\newtheorem{example}[theorem]{Example}

\theoremstyle{remark}
\newtheorem{remark}[theorem]{Remark}

\numberwithin{equation}{section}

\allowdisplaybreaks[4]

\begin{document}

\title{Blow-$ADE$ singularities and $\mu^*$-constant deformations}

\author{Christophe Eyral and Mutsuo Oka}

\address{C. Eyral, Institute of Mathematics, Polish Academy of Sciences, ul. \'Sniadeckich~8, 00-656 Warsaw, Poland}  
\email{cheyral@impan.pl} 
\address{M. Oka, 3-19-8 Nakaochiai, Shinjuku-ku, Tokyo 161-0032, Japan (Professor Emeritus of Tokyo Institute of Technology)}   
\email{okamutsuo@gmail.com}

\thanks{}

\subjclass[2020]{14M25, 14B05, 14J17, 32S55, 32S05}

\keywords{Complex surface singularities; blow-$ADE$ singularities; numerial control of the singularity type; Teissier's $\mu^*$-invariant; monodromy zeta-function.}
\begin{abstract}
We introduce a class of complex surface singularities \textemdash\ the blow-$ADE$ singularities \textemdash\ which are likely to be stable with respect to $\mu^*$-constant deformations. We prove such a stability property in several special cases. Here, we emphasize that we are not just considering deformation families for small values of the deformation parameter but families connecting any two elements in the $\mu^*$-constant stratum.
\end{abstract}

\maketitle

\markboth{Christophe Eyral and Mutsuo Oka}{Blow-$ADE$ singularities and $\mu^*$-constant deformations}

\section{Introduction}\label{intro}

The numerical control (i.e., the control by numerical invariants) 
of the ``singularity type'' of a singular point $O$ subjected to a certain deformation in a complex  analytic variety $V$ is a central question in singularity theory. The present paper is devoted to this problem in the case where $V$ is a surface $f=0$ of $\mathbb{C}^3$ and $O$ is an isolated singular point of $V$. More precisely, our goal here is to introduce a class of surface singularities \textemdash\ the blow-$ADE$ singularities \textemdash\ which are likely to be stable with respect to $\mu^*$-constant deformations.  

Roughly speaking, an isolated singular point $O$ of a surface $V\subseteq \mathbb{C}^3$ \textemdash\ which we shall assume to be the origin $\mathbf{0}\in\mathbb{C}^3$ \textemdash\ is a \emph{blow-$ADE$ singularity} if after an ordinary point blowing-up, the projective tangent cone $C$ of $V$ has only $ADE$ singularities, and if furthermore, there exists an integer $m\geq 1$ such that for any singular point $P$ of $C$, the Newton principal part, at $P$, of the strict transform of $V$ is of the form 
\begin{equation}\label{formintro}
h(x_2,x_3)+c\,x_1^m,
\end{equation}
 where $h(x_2,x_3)$ is the defining polynomial of $C$ at $P$ with respect to some local coordinates $\mathbf{x}=(x_1,x_2,x_3)$ and $c$ is a non-zero constant. The integer $m$ in \eqref{formintro} is called the \emph{blow-order} of the blow-$ADE$ singularity $(V,\mathbf{0})$. (For a more precise definition, see Definition~\ref{def-bNDsing}.) A typical example of such a singularity is given by a superisolated singularity, or more generally, by a L\^e--Yomdin singularity whose projective tangent cone has only $ADE$ singularities (see \cite{I,LeDT,Luengo,L-MH,A1}).

Throughout the paper, by a $\mu^*$-constant deformation of $f$, we mean a $\mu^*$-constant piecewise complex-analytic family $\{f_s\}$, $0\leq s\leq 1$, of complex-analytic functions $f_s$ such that $f_0=f$. Such a family is given by a continuous path 
\begin{equation*}
\gamma\colon s\in [0,1]\mapsto \gamma(s):=f_s\in\mathcal{M}(f)
\end{equation*}
in the $\mu^*$-constant stratum, $\mathcal{M}(f)$, of $f$ such that there are real numbers 
\begin{equation*}
0=s_0<s_1<\cdots<s_{j_0}=1,
\end{equation*}
and for each $0\leq j<j_0$, there is an open neighbourhood $U_j$ of the interval $[s_j,s_{j+1}]$ in $\mathbb{C}$ together with a complex-analytic function $\tilde\gamma_j$ on $U_j$ such that 
\begin{equation*}
\tilde\gamma_j\vert_{[s_j,s_{j+1}]}=\gamma\vert_{[s_j,s_{j+1}]}. 
\end{equation*}
(As usual, the $\mu^*$-constant stratum $\mathcal{M}(f)$ of $f$ consists of all function-germs whose Teissier's $\mu^*$-invariant (see \cite{Teissier2}) is the same as that of $f$.) Let us highlight that we are \emph{not} just considering deformations for small values of the parameter $s$ but deformations connecting $f$ to any other element in the path-connected component of $\mathcal{M}(f)$ containing $f$.

Finally, by ``stable singularity type'', we mean the following. If $f$ has a blow-$ADE$ singularity at $\mathbf{0}$ and if $\{f_s\}$ is a deformation of $f$, then we say that the singularity type of $\{f_s\}$ is \emph{stable} if each $f_s$ has a blow-$ADE$ singularity at $\mathbf{0}$ of the \emph{same type} as $f\equiv f_0$, which in turn means, firstly, that the projective tangent cones, $C_0$ and $C_s$, of $V_0=\{f_0=0\}$ and $V_s=\{f_s=0\}$ have the same number $k_0$ of singular points, and secondly, that we can make $k_0$ pairs $(P_0,P_s)\in \mbox{Sing}(C_0)\times \mbox{Sing}(C_s)$ involving all the singular points of $C_0$ and $C_s$ (here $\mbox{Sing}(\cdot)$ denotes the singular set) such that:
\begin{enumerate}
\item
the defining polynomials, $h_0$ and $h_s$, of $C_0$ and $C_s$ at $P_0$ and $P_s$ are given by the same expression in appropriate local coordinate systems $(x_{1},x_{2},x_{3})$ and $(x'_{1},x'_{2},x'_{3})$ at $P_0$ and $P_s$, respectively;
\item
the Newton principal parts of the strict transforms of $V_0$ and $V_s$ are given by
\begin{equation*}
h_0(x_{2},x_{3})+c_0\, x_{1}^m
\quad\mbox{and}\quad
h_s(x'_{2},x'_{3})+c_s\, {x'_{1}}^m
\end{equation*}
for the same blow-order $m$. 
\end{enumerate}
(For a more precise definition, see Definition \ref{def-sst}.)

Now, by a theorem of Teissier \cite[Th\'eor\`eme~3.9]{Teissier2}, we know that in a $\mu^*$-constant piecewise complex-analytic family of surface singularities, the embedded topological type of the family members is constant. In fact, the theorem of Teissier  applies to hypersurfaces and even says that the family is locally Whitney equisingular. We expect that for blow-$ADE$ singularities, the \emph{singularity type} (in the sense defined above) is also stable (see Conjecture~\ref{conj}). We give evidence for such a stability property in several special cases (see Theorem~\ref{mt1}). As an important corollary, we  get that if $f$ has a blow-$ADE$ singularity at~$\mathbf{0}$ of one of the specific types treated in Theorem \ref{mt1} and if $\mathcal{M}_0(f)$ is the path-connected component of $\mathcal{M}(f)$ containing $f$, then any $f'\in\mathcal{M}_0(f)$ also has a blow-$ADE$ singularity, and this blow-$ADE$ singularity is of the same type as that of $f$ (see Corollary \ref{cor}).

\section{Definition of blow-$ADE$ singularities}\label{sect-def}

Let $f(z_1,z_2,z_3)$ be an analytic function of order $d\geq 2$ in a neighbourhood of the origin $\mathbf{0}\in \mathbb{C}^3$, and let 
\begin{equation*}
f=\sum_{j\geq 0}f_{d+j}
\end{equation*}
 be its homogeneous decomposition. Here, $f_{d+j}$ is a homogeneous polymomial of degree $d+j$. We assume that $f(\mathbf{0})=0$ and that $f$ has an \emph{isolated} singularity at $\mathbf{0}$. We denote by $V(f)$ the (germ at~$\mathbf{0}$ of the) surface of $\mathbb{C}^3$ defined by $f$, and we write $C(f_{d+j})$ for the projective curve in $\mathbb{P}^2$ defined by the homogeneous polynomial $f_{d+j}$. We suppose that $f_d$ is \emph{reduced}, so that the curve $C(f_{d})$ has only \emph{isolated} singularities.
Without loss of generality, we may further assume that $f_d$ is convenient (i.e., up to a non-zero coefficient, $f_d$ has a monomial of the form $z_i^{\alpha_i}$ for each $1\leq i\leq 3$) and Newton non-degenerate on any face $\Delta$ of its Newton boundary if $\Delta$ is not the top-dimensional face. Indeed, it is easy to see that these two additional conditions are always simultaneously satisfied modulo a linear change of the coordinates $z_1,z_2,z_3$ (so, in particular, these conditions are not restrictive).

In order to define the blow-$ADE$ singularities, we consider the ordinary point blowing-up $\pi\colon X\to\mathbb{C}^3$ at~$\mathbf{0}$. Let $U_1:=\mathbb{P}^2\setminus \{z_1=0\}$ be the standard affine chart of $\mathbb{P}^2$ with coordinates $(z_2/z_1,z_3/z_1)$. Then, in the corresponding chart $X\cap (\mathbb{C}^3\times U_1)$ of $X$, with coordinates $\mathbf{y}\equiv(y_1,y_2,y_3):=(z_1,z_2/z_1,z_3/z_1)$, the pull-back $\pi^*f$ of $f$ by $\pi$ is given by
\begin{equation}\label{pull-back}
\pi^*f (\mathbf{y}) = y_1^{d}\bigg(f_{d}(1,y_2,y_3)+\sum_{j\geq 1} y_1^jf_{d+j}(1,y_2,y_3)\bigg).
\end{equation}
The first factor, $y_1^{d}$, corresponds to the exceptional divisor $E\simeq \mathbb{P}^2$, while the second one, $\pi^*f (\mathbf{y})/y_1^{d}$, represents the strict transform of $V(f)$. Put
\begin{equation}\label{ftilda}
\tilde f(\mathbf{y}):=\pi^*f (\mathbf{y})/y_1^{d}.
\end{equation}

Let $\text{Sing}(C(f_d))$ be the singular locus of $C(f_d)$, and let $m$ be an integer $\geq 1$.

\begin{definition}\label{def-bNDsing}
We say that $f$ has a \emph{blow-$ADE$} singularity of \emph{blow-order} $m$ at $\mathbf{0}$ if for any singular point $P\in \text{Sing}(C(f_d))$, there exists an ``admissible'' coordinate chart $(U_P,\mathbf{x}=(x_1,x_2,x_3))$ of $X$ at $P$ such that:
\begin{enumerate}
\item
the defining polynomial, $h(x_2,x_3)$, of $C(f_d)$ at $P$ in the coordinates $\mathbf{x}$ is: 
\vskip 1mm
\begin{enumerate}
\item[$\bullet$]
either $x_2^2+x_3^{n+1}$, where $n\geq 1$ ($A_n$-type);
\vskip 1mm
\item[$\bullet$]
or $x_2^2x_3+x_3^{n-1}$, where $n\geq 4$ ($D_n$-type);
\vskip 1mm
\item[$\bullet$]
or $x_2^3+x_3^4$ ($E_6$-type) or
$x_2^3+x_2x_3^3$ ($E_7$-type) or
$x_2^3+x_3^5$ ($E_8$-type);
\end{enumerate}
\vskip 1mm
(up to non-zero coefficients);
\vskip 1mm
\item
the Newton principal part of $\tilde f$ at $P$ in the coordinates $\mathbf{x}$ is of the form 
\begin{equation*}
h(x_2,x_3)+c\, x_1^m, 
\end{equation*}
where $c$ is a non-zero constant.
\end{enumerate}
\end{definition} 

Here, by an ``admissible'' coordinate chart, we mean an analytic chart $(U_P,\mathbf{x}=(x_1,x_2,x_3))$ such that $x_1=y_1$ and $(x_2,x_3)$ is an analytic coordinate change of $(y_2,y_3)$.

\begin{example}
A typical example of a blow-$ADE$ singularity of blow-order $m$ is given by a $m$-L\^e--Yomdin singularity \textemdash\ i.e., a singularity defined by a function $f=f_{d}+f_{d+m}+\cdots$ for which the intersection $\mbox{Sing}(C(f_d))\cap C(f_{d+m})$ is empty (see \cite{I,LeDT,Luengo,L-MH,A1}) \textemdash\ such that $C(f_d)$ has only $ADE$-singularities. Note that the condition $\mbox{Sing}(C(f_d))\cap C(f_{d+m})=\emptyset$ implies that $f_d$ is reduced and that $f$ has an isolated singularity at~$\mathbf{0}$ (see \cite[Theorem~2]{L-MH}). Of particular interest are the $1$-L\^e--Yomdin singularities, which are exactly Luengo's superisolated singularities. 
\end{example}

\begin{example}
Note that \emph{not} all blow-$ADE$ singularities are L\^e--Yomdin singularities. For example, consider any blow-$ADE$ singularity of blow-order $m$ of the form $f=f_d+f_{d+m}$. Then the singularity defined by 
\begin{equation*}
g:=f_d+c_1\ell f_d+c_2\ell^2 f_d+\cdots+c_{m-1}\ell^{m-1}f_d+f_{d+m}
\end{equation*}
 is also a blow-$ADE$ singularity of blow-order $m$ as long as the coefficients $c_1,\ldots,c_{m-1}$ are generically chosen. Here, $\ell$ is a linear form. However, $g$ is not a L\^e--Yomdin singularity as the intersection $\mbox{Sing}(C(f_d))\cap C(\ell f_d)$ is not empty.
\end{example}

In this paper, we shall focus on some specific types of blow-$ADE$ singularities which we define below.

\begin{definition}
A blow-$ADE$ singularity $f$ is called a \emph{pure blow-$A_1$ singularity} if each $P\in \text{Sing}(C(f_d))$ is of $A_1$-type. It is called a \emph{blow-$A$ singularity} if each $P\in \text{Sing}(C(f_d))$ is of $A_n$-type for some integer $n$ depending on $P$ (in particular, if $P'$ is another singular point with $P'\not=P$, then $P'$ can be of $A_{n'}$-type for $n'\not=n$). Finally, $f$ is called an \emph{even} blow-$A$ singularity if each singular point $P$ is of $A_n$-type for some \emph{even} integer $n\geq 2$ depending on $P$.
\end{definition}

To conclude this section, let us clarify what we exactly mean for two blow-$ADE$ singularities to have the ``same type''.

\begin{definition}\label{def-sst}
Let $f'$ be another analytic function of the same order $d$ as $f$, and let $f'=\sum_{j\geq 0}f'_{d+j}$ be its homogeneous decomposition.
Assume that both $f$ and $f'$ have a blow-$ADE$ singularity at $\mathbf{0}$.
Then we say that $f$ and $f'$ are two blow-$ADE$ singularities of the \emph{same type} if the following two conditions are satisfied:
\begin{enumerate} 
\item
$C(f_d)$ and $C(f'_d)$ have the same number $k_0$ of singular points, and we can make pairs $(P_k,P'_k)\in \mbox{Sing}(C(f_d))\times \mbox{Sing}(C(f'_d))$ such that for each $1\leq k\leq k_0$ the defining polynomials, $h(x_2,x_3)$ and $h'(x'_2,x'_3)$, of the curves $C(f_d)$ and $C(f'_d)$ at the points $P_k$ at $P'_k$, respectively, are given by the same expression in appropriate local coordinate systems $\mathbf{x}=(x_1,x_2,x_3)$ and $\mathbf{x}'=(x'_1,x'_2,x'_3)$ at $P_k$ and $P'_k$, respectively;
\item
the Newton principal parts of $\tilde f$ and $\tilde f'$ (see \eqref{ftilda}) are given by
\begin{equation*}
h(x_2,x_3)+c\, x_1^m
\quad\mbox{and}\quad
h'(x'_2,x'_3)+c'{x_1'}^m
\end{equation*}
(in particular, $f$ and $f'$ have the same blow-order $m$).
\end{enumerate}
\end{definition}

\section{$\mu^*$-constant deformations of blow-$ADE$ singularities}

We continue with the notations and assumptions set in the preamble of section \ref{sect-def}.

Here is our main result.

\begin{theorem}\label{mt1}
Assume that $f$ has, at $\mathbf{0}$, a blow-$ADE$ singularity of one of the following three  specific types:
\begin{enumerate}
\item
a pure blow-$A_1$ singularity of blow-order $m\geq 1$;
\vskip 1mm
\item
an even blow-$A$ singularity of blow-order $m=2$;
\vskip 1mm
\item
a blow-$ADE$ singularity of blow-order $m=1$.
\end{enumerate} 
In each case, if $\{f_s\}$, $0\leq s\leq 1$, is any $\mu^*$-constant piecewise complex-analytic family starting at $f$ (i.e., $f_0=f$), then each member $f_s$ of the family defines, at $\mathbf{0}$, a blow-$ADE$ singularity of the same type as that of $f$.
\end{theorem}

In other words, any blow-$ADE$ singularity of type (1), (2) or (3) is stable with respect to $\mu^*$-constant piecewise complex-analytic deformations. We expect that any blow-$ADE$ singularity is stable under such deformations. More precisely, we propose the following conjecture.

\begin{conjecture}\label{conj}
The conclusion of Theorem \ref{mt1} still holds true if $f$ defines a general blow-$ADE$ singularity of any blow-order $m\geq1$ at $\mathbf{0}$.
\end{conjecture}

Again, note that both Theorem \ref{mt1} and Conjecture \ref{conj} apply to piecewise complex-analytic families connecting \emph{any} two elements in a path-connected component of the $\mu^*$-constant stratum, and not just to small deformation families.

In particular, combined with \cite[Theorem 5.4]{EO2}, Theorem \ref{mt1} has the following important corollary.

\begin{corollary}\label{cor}
Let $f$ be a blow-$ADE$ singularity of one of the three specific types (1), (2) or (3) that appear in Theorem \ref{mt1}, and let $\mathcal{M}_0(f)$ be the path-connected component of $\mathcal{M}(f)$ containing $f$. (We recall that $\mathcal{M}(f)$ is the $\mu^*$-constant stratum of $f$.) Then, any $f'\in\mathcal{M}_0(f)$ defines a blow-$ADE$ singularity of the same type as that of $f$.
\end{corollary}

\begin{proof}
If $f'\in\mathcal{M}_0(f)$, then, by \cite[Theorem 5.4]{EO2}, there exists a $\mu^*$-constant piecewise complex-analytic family connecting $f$ and $f'$. Now apply Theorem \ref{mt1}.
\end{proof}

The rest of the paper is devoted to the proof of Theorem \ref{mt1}. In each case, (1), (2) and (3), the argument is quite different and depends on the specific type of the blow-$ADE$ singularity that we consider.

\section{Proof of Theorem \ref{mt1}}

We start with the following two important observations, which are already proved in \cite{EO} (see Lemmas 5.2, 5.3 and the end of \S 2 there).

\begin{lemma}[see \cite{EO}]\label{ml2}
For any $0\leq s\leq 1$, the initial polynomial $f_{s,d}$ of $f_s$ is reduced. In particular, the projective curve $C(f_{s,d})$ defined by $f_{s,d}$ has only isolated singularities. Moreover, the total Milnor number $\mu^{\text{\tiny \emph{tot}}}(C(f_{s,d}))$ of $C(f_{s,d})$, the local embedded topological types of the singularities of $C(f_{s,d})$, and the (global)  topological type of the pair $(\mathbb{P}^2,C(f_{s,d}))$ do not depend on $s$.
\end{lemma}

Here, by the ``total Milnor number'' of $C(f_{s,d})$, we mean the sum of the local Milnor numbers at the singular points of $C(f_{s,d})$.

Let $k_0$ (respectively, $\mu$) denote the common number of singular points (respectively, the common total Milnor number) of the curves $C(f_{s,d})$, $0\leq s\leq 1$. 

\begin{lemma}[see \cite{EO}]\label{ml3-0}
For any fixed real number $s$, $0\leq s\leq 1$, let $P_1,\ldots,P_{k_0}$ be the singular points of $C(f_{s,d})$, and for each $P_k$, let $B_\varepsilon (P_k)$ be a ``small'' ball centred at $P_k$.
Then the monodromy zeta-function $\zeta_{f_s,\mathbf{0}}(t)$ of $f_s$ at $\mathbf{0}$ is given by 
\begin{equation*}
\zeta_{f_s,\mathbf{0}}(t)=(1-t^{d})^{-d^2+3d-3+\mu} \
\prod_{k=1}^{k_0}\zeta_{\pi^* f_s\vert_{B_\varepsilon(P_k)}}(t),
\end{equation*}
where $\zeta_{\pi^* f_s\vert_{B_\varepsilon(P_k)}}(t)$ is the zeta-function of the fibration $\pi^* f_s\vert_{B_\varepsilon(P_k)}$.
\end{lemma}

Note that Lemma \ref{ml2} implies the next statement which describes the singularities of $\tilde f_s$.

\begin{lemma}\label{38}
The strict transform $V(\tilde f_s)$ of $V(f_s)$ has only isolated singularities.
\end{lemma}

\begin{proof}
As $f_s$ has an isolated singularity at $\mathbf{0}$, in a sufficiently small neighbourhood of the exceptional divisor $E=\{y_1=0\}$, the singular points of $V(\tilde f_s)$ necessarily lie on $E$. Besides, we easily check that the singular points of $V(\tilde f_s)$ on $E$ coincide with those of $f_{s,d}(1,y_2,y_3)=0$ on $E$. As the latter are nothing but the singular points of $C(f_{s,d})$, Lemma \ref{38} follows from Lemma \ref{ml2}.
\end{proof}

Now we consider each case (1), (2) and (3) of the theorem separately.

\subsection{Proof of Theorem \ref{mt1} -- case (1)}\label{sect31}
Pick any $s$, $0\leq s\leq 1$. The first key observation is the following lemma.

\begin{lemma}\label{ml1}
For any singular point $P_k \in\text{\emph{Sing}}(C(f_{s,d}))$, there exist admissible coordinates $\mathbf{x}=(x_1,x_2,x_3)$ near $P_k$ in which the Newton principal part of $\pi^* f_s$ is of the form 
\begin{equation}\label{specialform}
\pi^* f_s=x_1^{d}(x_2^2+x_3^2+c\,x_1^{\ell_k}),
\end{equation}
where $c\in\mathbb{C}^*$ and $\ell_k\in\mathbb{N}^*$.
\end{lemma}

\begin{proof}
As $P_k$ is a node (see Lemma \ref{ml2}), there are admissible coordinates $\mathbf{x}=(x_1,x_2,x_3)$ near $P_k$ in which $\pi^* f_s$ is of the form 
\begin{equation}\label{gen-not-1}
\pi^* f_s=x_1^{d}\tilde f_s=x_1^{d}\bigg(h(x_2,x_3)+\sum_{j\geq 1}x_1^j g_j(x_2,x_3)\bigg)
\end{equation}
where $h(x_2,x_3)=x_2^2+x_3^2$ is the defining polynomial of $C(f_{s,d})$ near $P_k$ (see \eqref{pull-back} and \eqref{ftilda}). Here, the $g_j$'s are not necessarily homogeneous.

\begin{claim}\label{claimjm}
There exists $j\geq 1$ such that the term $x_1^j g_j(x_2,x_3)$ in \eqref{gen-not-1} contains, up to coefficients,  at least one monomial of the form $x_1^j$, $x_1^jx_2$ or $x_1^jx_3$. 
\end{claim}

\begin{proof}
We argue by contradiction. Suppose that the assertion is not true. Then, for any $j\geq 1$, all the monomials of $x_1^j g_j(x_2,x_3)$ are of the form $x_1^{j} x_2^{j_2}x_3^{j_3}$ with $j_2+ j_3\geq 2$. In particular, this implies that the critical locus of $\tilde f_s$ contains the line $x_2=x_3=0$, and hence, that it is $1$-dimensional. But this, in turn, implies that the critical locus of $f_s$ is $1$-dimensional too, which is a contradiction.
\end{proof}

Let $j_0$ denote the smallest integer $j$ satisfying Claim \ref{claimjm}. If $x_1^{j_0} g_{j_0}(x_2,x_3)$ contains a term of the form $c\,x_1^{j_0}$ with $c\not=0$, then we take $\ell_k={j_0}$ and the lemma is proved. Indeed, in this case, since $P_k$ is a node, all (the vertices corresponding to) the other monomials of $\pi^* f_s$ are strictly above the Newton boundary of $x_1^{d}(x_2^2+x_3^2+c\,x_1^{j_0})$. 

Now, if $x_1^{j_0} g_{j_0}(x_2,x_3)$ does not contain the term $x_1^{j_0}$ with a non-zero coefficient, then, by Claim \ref{claimjm}, it contains the monomials 
\begin{equation*}
a_{j_0}x_1^{j_0}x_2
\quad\mbox{and}\quad
b_{j_0}x_1^{j_0}x_3,
\end{equation*}
 where, at least, either $a_{j_0}$ or $b_{j_0}$ is non-zero.  To prove the Lemma \ref{ml1}, we must eliminate these monomials. For this purpose, we apply the following procedure. We consider the coordinates change $\mathbf{x}=(x_1,x_2,x_3)\mapsto\mathbf{x}'=(x_1,x'_2,x'_3)$ defined by 
\begin{equation}\label{lcc}
x'_2:=x_2+(a_{j_0}/2) x_1^{j_0}
\quad\mbox{and}\quad
x'_3:=x_3+(b_{j_0}/2) x_1^{j_0}.
\end{equation}
In these new coordinates, the pull-back $\pi^*f_s$ is written as
\begin{equation*}
\pi^*f_s=x_1^{d}({x'_2}^2+{x'_3}^2+c_{j_0}\, x_1^{2j_0}+R(\mathbf{x}')),
\end{equation*}
where $c_{j_0}$ is a (possibly zero) constant and $R(\mathbf{x}')$ is an analytic function. Note that $R(\mathbf{x}')$ may contain monomials of the form $x_1^j$, $x_1^jx'_2$ and $x_1^jx'_3$ only for $j>j_0$. 
If $c_{j_0}\not=0$ and $R(\mathbf{x}')$ contains only monomials which are strictly above the Newton boundary of 
\begin{equation*}
{x'_2}^2+{x'_3}^2+c_{j_0}\, x_1^{2j_0},
\end{equation*}
then Lemma \ref{ml1} is proved. Otherwise, we repeat the procedure with respect to the new coordinates $\mathbf{x}'$, that is, we apply a new change of coordinates $\mathbf{x}'\mapsto \mathbf{x}''$ (similar to the change $\mathbf{x}\mapsto \mathbf{x}'$ defined in \eqref{lcc}) in order to eliminate the new unwanted terms.

\begin{claim}\label{claimstop}
After a finite number of steps, this procedure stops, that is, at some point, we end up with admissible coordinates in which the Newton principal part of $\pi^* f_s$ at $P_k$ takes the form \eqref{specialform}.
\end{claim}

\begin{proof}
Claim \ref{claimstop} is clear if a monomial of the form $c\,x_1^j$ with $c\not=0$ and $j_0<j\leq 2j_0$ appears in $\tilde f$. If $\tilde f$ does not include such a monomial, then, by Claim \ref{claimjm}, there exists $j_1$ such that $\tilde f$ contains at least one monomial of the form $x_1^{j_1}$, $x_1^{j_1} x'_2$ or $x_1^{j_1} x'_3$ with a non-zero coefficient and does not include any term of the form $c\,x_1^j$, $c\not=0$, for $j\leq 2j_0$. Let $j_{\mbox{\tiny min}}$ be the smallest integer among all such $j_1$'s. If $x_1^{j_{\mbox{\tiny min}}}$ appears with a non-zero coefficient, then we conclude as above taking $\ell_k=j_{\mbox{\tiny min}}$. Otherwise, Claim \ref{claimstop} follows immediately from the next statement.

\begin{claim}\label{ass35}
If $\tilde f$ does not include the term $x_1^{j_{\mbox{\emph{\tiny min}}}}$ with a non-zero coefficient, then the Milnor number, $\mu_{P_k}(\tilde f_s)$, of $\tilde f_s$ at $P_k$ satisfies the following inequality:
\begin{equation*}
\mu_{P_k}(\tilde f_s)\geq 2j_{\mbox{\emph{\tiny min}}}-1.
\end{equation*}
\end{claim}

That Claim \ref{claimstop} follows from Claim \ref{ass35} is clear. Indeed, if $j_{\mbox{\tiny min}}$ becomes arbitrarily large, so does $\mu_{P_k}(\tilde f_s)$, which is a contradiction.
\end{proof}

\begin{proof}[Proof of Claim \ref{ass35}]
Consider the deformation family
\begin{equation*}
\tilde f_{s,\tau}:=\tilde f_s+\tau x_1^{2j_{\mbox{\tiny min}}}
\end{equation*}
for $\tau\geq 0$ small enough. Clearly $\tilde f_{s,\tau}$, $\tau\not=0$, is Newton non-degenerate. Thus, by \cite[Th\'eor\`eme 1.10]{K}, its Milnor number at $P_k$ is given by $\mu_{P_k}(\tilde f_{s,\tau})=2j_{\mbox{\tiny min}}-1$. It then follows from the upper-semicontinuity of the Milnor number that 
\begin{equation*}
\mu_{P_k}(\tilde f_s)\geq \mu_{P_k}(\tilde f_{s,\tau})=2j_{\mbox{\tiny min}}-1.\qedhere
\end{equation*}
\end{proof}
This completes the proof of Lemma \ref{ml1}.
\end{proof}

Together with Lemma \ref{ml1}, the second crucial observation to prove Theorem \ref{mt1} (in case~(1)) is the following lemma.

\begin{lemma}\label{ml3}
For all $1\leq k\leq k_0$, the integer $\ell_k$ defined in Lemma \ref{ml1} is equal to $m$ and the monodromy zeta-function $\zeta_{f_s,\mathbf{0}}(t)$ is given by
\begin{equation*}
\zeta_{f_s,\mathbf{0}}(t)=(1-t^{d})^{-d^2+3d-3+\mu} (1-t^{d+m})^{-\mu}.
\end{equation*}
\end{lemma}

\begin{proof}
As $f$ has a pure blow-$A_1$ singularity at $\mathbf{0}$, we have $k_0=\mu$. Moreover, by Lemma \ref{ml3-0}, we know that
\begin{equation*}
\zeta_{f_s,\mathbf{0}}(t) = (1-t^d)^{-d^2+3d-3+\mu} \cdot \prod_{k=1}^{\mu} \zeta_{\pi^* f_s\vert_{B_\varepsilon(P_k)}}(t),
\end{equation*}
and since $\pi^* f_s\vert_{B_\varepsilon(P_k)}$ is Newton non-degenerate (see Lemma \ref{ml1}), we have
\begin{equation*}
\zeta_{\pi^* f_s\vert_{B_\varepsilon(P_k)}}(t) = (1-t^{d+\ell_k})^{-1}
\end{equation*}
(see \cite[Theorem (4.1)]{V}).
Altogether,
\begin{equation*}
\zeta_{f_s,\mathbf{0}}(t)=(1-t^{d})^{-d^2+3d-3+\mu} \prod_{k=1}^{\mu}(1-t^{d+\ell_k})^{-1}.
\end{equation*}
Now, since $\zeta_{f_s,\mathbf{0}}(t)=\zeta_{f_0,\mathbf{0}}(t)$ (see \cite[Th\'eor\`eme~3.9]{Teissier2}) and since the blow-order of the pure blow-$A_1$ singularity $f_0\equiv f$ is $m$, it follows immediately that $\ell_k=m$ for all $1\leq k\leq \mu$.
\end{proof}

Now, case (1) of Theorem \ref{mt1} follows immediately from Lemmas \ref{ml1} and \ref{ml3}.

\subsection{Proof of Theorem \ref{mt1} -- case (2)}\label{sect-32}

Again, pick any $s$, $0\leq s\leq 1$. 
By Lemma \ref{ml3-0}, the first zeta-multiplicity factor $\zeta^{(1)}_{f_s,\mathbf{0}}(t)$ of the zeta-function $\zeta_{f_s,\mathbf{0}}(t)$ (see Appendix \ref{AppA1} for the definition) is given by 
\begin{equation*}
\zeta^{(1)}_{f_s,\mathbf{0}}(t)=(1-t^{d})^{-d^2+3d-3+\mu}.
\end{equation*}
As $f_0\equiv f$ has an \emph{even} blow-$A$ singularity of blow-order $2$, each singular point of $C(f_{0,d})$ contributes to the second zeta-multiplicity factor $\zeta^{(2)}_{f_0,\mathbf{0}}(t)$ of $\zeta_{f_0,\mathbf{0}}(t)$ by the term $(1-t^{d+2})^{+1}$ (for the definition of $\zeta^{(2)}_{f_0,\mathbf{0}}(t)$, see Appendix~\ref{AppA1}). This can be easily seen using \cite[Lemma 3.2 and Theorem 3.7]{O2} and \cite[Theorem (4.1)]{V}. So, altogether, we have
\begin{equation}\label{zfdp2}
\zeta^{(2)}_{f_0,\mathbf{0}}(t)=(1-t^{d+2})^{k_0}.
\end{equation} 

Now, take any singular point $P_k\in\mbox{Sing}(C(f_{s,d}))$, $s\not=0$, and consider admissible coordinates $\mathbf{x}=(x_1,x_2,x_3)$ near $P_k$ in which $\pi^* f_s$ is of the form 
\begin{equation}\label{exppbicx}
\pi^* f_s=x_1^{d}\bigg(h(x_2,x_3)+\sum_{j\geq 1}x_1^j g_j(x_2,x_3)\bigg)
\end{equation}
where $h(x_2,x_3)$ is the defining polynomial of $C(f_{s,d})$ near $P_k$. In the present case, $h(x_2,x_3)=x_2^2+x_3^{n+1}$ for some even integer $n$ depending on $P_k$ (see Lemma \ref{ml2}).

\begin{claim}\label{claim310}
The first term, $x_1g_1$, under the sum symbol in \eqref{exppbicx} does not contain the linear term~$x_1$. (As above, by $x_1$ we mean $x_1$ up to a non-zero coefficient.)
\end{claim}

\begin{proof}
We argue by contradiction. If $g$ contains $x_1$, then 
\begin{equation*}
\pi^* f_s=x_1^{d}(h(x_2,x_3)+a\, x_1+\cdots),\ a\not=0,
\end{equation*}
and the latter expression is Newton non-degenerate. (Here, the dots ``$\cdots$'' stand for the higher-order terms.) Then, by \cite[Theorem (4.1)]{V}, we easily check that the zeta-function $\zeta_{\pi^* f_s\vert_{B_\varepsilon(P_k)}}(t)$, and hence $\zeta_{f_s,\mathbf{0}}(t)$, contain the factor $(1-t^{d+1})^{-1}$. On the other hand, $\zeta_{f_0,\mathbf{0}}(t)$ does not contain this factor. This is a contradiction, as the zeta-functions $\zeta_{f_0,\mathbf{0}}(t)$ and $\zeta_{f_s,\mathbf{0}}(t)$ must be identical by \cite[Th\'eor\`eme 3.9]{Teissier2}.
\end{proof}

Let $f_s=\sum_{j\leq 0}f_{s,d+j}$ be the homogeneous decomposition of $f_s$.
By Claim \ref{claim310}, the pull-back $\pi^*(f_{s,d}+f_{s,d+1})=x_1^d(h(x_2,x_3)+ x_1\, g_1(x_2,x_3))$ is of the form
\begin{equation*}
x_1^d\big(x_2^2+x_3^{n+1}+A(x_2,x_3)\, x_1x_2+B(x_3)\, x_1x_3^q\big)
\end{equation*} 
for some integer $1\leq q\leq (n+1)/2$ and some analytic functions $A(x_2,x_3)$ and $B(x_3)$. 
Put 
\begin{equation*}
x_2':=x_2+x_1 A(x_2,x_3)/2.
\end{equation*} 
By the implicit function theorem, in a small neighbourhood of the origin, the variable $x_2$ can be written as $x_2=\phi(x_1,x_2',x_3)$, and in the new coordinates $\mathbf{x}':=(x_1,x_2',x_3)$, the pull-back $\pi^*(f_{s,d}+f_{s,d+1}+f_{s,d+2})$ takes the form
\begin{align}\label{DNPP}
x_1^d({x_2'}^2+x_3^{n+1}+x_1^2 (g_2(\phi(\mathbf{x}'),x_3)-A^2(\phi(\mathbf{x}'),x_3)/4)+B(x_3)\, x_1x_3^q).
\end{align}  
It follows that if the Newton principal part $\mbox{NPP}(\pi^* f_s)$ of $\pi^* f_s$ contains the term $a x_1^2$ for some $a\not=0$, then it is of the form 
\begin{equation}\label{provi}
x_1^d({x_2'}^2+x_3^{n+1}+a\, x_1^2+b\, x_1x_3^q),
\end{equation} 
where $b:=B(0)$. On the other hand, if it does not contain such a term (i.e., if $a=0$), then $\mbox{NPP}(\pi^* f_s)$  is written as
\begin{equation}\label{nppi2}
x_1^d({x_2'}^2+x_3^{n+1}+b\, x_1x_3^q+R(\mathbf{x}')),
\end{equation} 
where $R(\mathbf{x}')$ is some polynomial of order $>d+2$ (the exact form of $R(\mathbf{x}')$ does not matter). 

\begin{claim}\label{(i)}
If $a\not=0$ and $b=0$, then $P_k$ contributes to the second zeta-multiplicity factor $\zeta^{(2)}_{f_s,\mathbf{0}}(t)$ by the term $(1-t^{d+2})^{+1}$; moreover, we have 
\begin{equation*}
\deg \zeta_{\pi^* f_s\vert_{B_\varepsilon(P_{s,k})}}(t)=\deg \zeta_{\pi^* f_0\vert_{B_\varepsilon(P_{0,k})}}(t).
\end{equation*}
\end{claim}

By abuse of language, in this claim as well as in the following ones, by $a\not=0$ (respectively, $a=0$), we mean that the Newton principal part of $\pi^* f_s$ \emph{contains} (respectively, \emph{does not contain}) the quadratic term $x_1^2$ with a non-zero coefficient.

\begin{proof}
Claim \ref{(i)} is clear as if $a\not=0$ and $b=0$, then the Newton principal part of $\pi^* f_s$ is exactly of the same form as that of $\pi^* f_0$.
\end{proof}

\begin{claim}\label{(ii)}
If $a\not=0$ and $b\not=0$, then $P_k$ also contributes to the second zeta-multiplicity factor $\zeta^{(2)}_{f_s,\mathbf{0}}(t)$ by the term $(1-t^{d+2})^{+1}$; on the other hand, this time, we have 
\begin{equation*}
\deg \zeta_{\pi^* f_s\vert_{B_\varepsilon(P_{s,k})}}(t)>\deg \zeta_{\pi^* f_0\vert_{B_\varepsilon(P_{0,k})}}(t).
\end{equation*}
\end{claim}

\begin{proof}
This follows again from a standard application of the classical Varchenko formula \cite[Theorem (4.1)]{V}.
If we only want to compute the degrees of the zeta-functions, then we may also, in a slightly simpler way, use the following lemma, which is of course an immediate consequence of \cite[Theorem (4.1)]{V} (see also \cite[Chapter~III, Theorem (5.3)]{O1}).

\begin{lemma}[Varchenko]\label{lemmaNN}
If $\phi$ is a Newton non-degenerate polynomial function in $x_1,x_2,x_3$ such that $\phi(\mathbf{0})=0$, then 
\begin{equation*}
\deg \zeta_{\phi,\mathbf{0}}(t)=-1-\nu(\phi),
\end{equation*}
where $\nu(\phi):=\sum_{I\subseteq \{1,2,3\}} (-1)^{3-|I|}\, |I|!\, \text{\emph{Vol}}_{|I|}(\Gamma_{\! -}(\phi^I))$ is the Newton number of $\phi$.
\end{lemma}

Here, $\Gamma_{\! -}(\phi^I)$ denotes the cone over the Newton boundary $\Gamma(\phi^I)$ of $\phi^I:=\phi\vert_{\mathbb{C}^I}$, where $\mathbb{C}^I:=\{(x_1,x_2,x_3)\in\mathbb{C}^3 \mid x_i=0 \mbox{ if } i\notin I\}$. The symbol $|I|$ denotes the cardinality of $I$, and $\text{Vol}_{|I|}$ is the $|I|$-dimensional Euclidean volume. For $I=\emptyset$, the subset $\Gamma(\phi^\emptyset)$ reduces to the origin of $\mathbb{R}^3$, and we set $\text{Vol}_{0}(\Gamma(\phi^\emptyset))=1$.

As the Newton boundary of $\pi^* f_0\vert_{B_\varepsilon(P_{0,k})}$ has a single top-dimensional face, given by the vertices $(d+2,0,0)$, $(d,2,0)$, $(d+1,0,n+1)$, and since $\pi^* f_0$ is Newton non-degenerate, Lemma \ref{lemmaNN} shows that
\begin{equation*}
\deg \zeta_{\pi^* f_0\vert_{B_\varepsilon(P_{0,k})}}(t)=-n(d+2).
\end{equation*}
For $s\not=0$, the Newton boundary of $\pi^* f_s\vert_{B_\varepsilon(P_{s,k})}$ has two top-dimensional faces \textemdash\ one is given by the vertices $(d+2,0,0)$, $(d,2,0)$, $(d+1,0,q)$, while the second one is given by $(d,2,0)$, $(d+1,0,q)$, $(d,0,n+1)$. Since $\pi^* f_s\vert_{B_\varepsilon(P_{s,k})}$ is also Newton non-degenerate, applying Lemma \ref{lemmaNN} gives
\begin{equation*}
\deg \zeta_{\pi^* f_s\vert_{B_\varepsilon(P_{s,k})}}(t)=
-(d+2)(q-1)-((d+1)(n+1)-qd).
\end{equation*}
Altogether, $\deg \zeta_{\pi^* f_s\vert_{B_\varepsilon(P_{s,k})}}(t)-\deg \zeta_{\pi^* f_0\vert_{B_\varepsilon(P_{0,k})}}(t) = -2q+n+1 > 0$ as desired.
\end{proof}

\begin{claim}\label{(iii)}
If $a=b=0$ or if $a=0$, $b\not=0$ and $q\geq 2$, then $P_k$ cannot produce the factor $(1-t^{d+2})^{+1}$ in $\zeta^{(2)}_{f_s,\mathbf{0}}(t)$.
\end{claim}

\begin{proof}
For the above values of $a$, $b$ and $q$, the Newton principal part of $\pi^* f_s$ is given by the expression \eqref{nppi2}. As this expression contains only one term of degree $d+2$ (namely, $x_1^d{x_2'}^2$) and since the corresponding vertex $(d,2,0)$ in the Newton boundary of $\pi^* f_s$ is not located on the ``$x_1$-axis'', it follows from \cite[Lemma 17]{O3} that the zeta-function $\zeta_{\pi^* f_s\vert_{B_\varepsilon(P_{s,k})}}(t)$ cannot contain the factor $(1-t^{d+2})$. Now the claim follows from Lemma \ref{ml3-0}.
\end{proof}

\begin{claim}\label{(iv)}
Finally, the case $a=0$, $b\not=0$ and $q=1$ appears as a special case of the situation discussed in Claim \ref{(ii)}.
\end{claim}

\begin{proof}
Indeed, if $a=0$, $b\not=0$ and $q=1$, then after the coordinates change $x_3\mapsto x_3':=x_3-x_1$, the Newton principal part of $\pi^* f_s$ takes the form
\begin{equation*}
x_1^d({x_2'}^2+{x_3'}^{n+1}+bx_1^2+b\, x_1x_3').\qedhere
\end{equation*} 
\end{proof}

We can now easily complete the proof of the theorem. Again, note that the integers $n$, $a$, $b$ and $q$ depend on the singular point $P_k$ of $C(f_{s,d})$. As we are going to consider all the $P_k$'s simultaneously in what follows, in order to avoid any confusion we shall write $n_k$, $a_k$, $b_k$ and $q_k$ instead of $n$, $a$, $b$ and $q$. Let $\mathcal{E}_1$ be the set of $P_k$'s for which $a_k\not=0$ and $b_k=0$ (situation of Claim \ref{(i)}); let $\mathcal{E}_2$ be the set of $P_k$'s for which $a_k\not=0$ and $b_k\not=0$ (situation of Claim \ref{(ii)}); finally, let $\mathcal{E}_3$ be the set of $P_k$'s for which $a_k=b_k=0$ or for which $a_k=0$, $b_k\not=0$ and $q_k\geq 2$ (situation of Claim \ref{(iii)}). To prove the theorem, it suffices to show that $\mathcal{E}_2$ and $\mathcal{E}_3$ are empty. We start with $\mathcal{E}_2$, and we argue by contradiction. Suppose that $\mathcal{E}_2\not=\emptyset$. If $\mathcal{E}_3=\emptyset$, then, by Claim \ref{(ii)} and Lemma \ref{ml3-0},
\begin{equation*}
\deg \zeta_{f_s,\mathbf{0}}(t)>\deg \zeta_{f_0,\mathbf{0}}(t),
\end{equation*}  
which is a contradiction (we recall that, by \cite[Th\'eor\`eme 3.9]{Teissier2}, $\zeta_{f_0,\mathbf{0}}(t)=\zeta_{f_s,\mathbf{0}}(t)$). If $\mathcal{E}_3\not=\emptyset$, then, since the points of $\mathcal{E}_3$ cannot give the factor $(1-t^{d+2})^{+1}$ (see Claim \ref{(iii)}), and since each point of $\mathcal{E}_1$ or $\mathcal{E}_2$ gives exactly the factor $(1-t^{d+2})^{+1}$ (see Claims \ref{(i)} and \ref{(ii)}), it follows from Lemma \ref{ml3-0} that
\begin{equation}\label{contradiction}
\zeta^{(2)}_{f_s,\mathbf{0}}(t)\not=(1-t^{d+2})^{k_0}\overset{\eqref{zfdp2}}{=}\zeta^{(2)}_{f_0,\mathbf{0}}(t),
\end{equation}  
which is a contradiction. 

Let us now show that $\mathcal{E}_3=\emptyset$ knowing that $\mathcal{E}_2=\emptyset$. Again we argue by contradiction. Suppose that $\mathcal{E}_3\not=\emptyset$. As $\mathcal{E}_2=\emptyset$, the singular points of $C(f_{s,d})$ which are not in $\mathcal{E}_3$ are in $\mathcal{E}_1$, but then Claims \ref{(i)} and \ref{(iii)} lead again to the contradiction \eqref{contradiction}.

This completes the proof of Theorem \ref{mt1} in case (2).

\begin{remark}
Claim \ref{(ii)} shows why we have assumed that $n$ is even. Indeed, suppose that $n$ is odd. Then $n+1=2q_0$ for some $q_0$, and if it arises that $q=q_0$ for some $P_k$, then the two top-dimensional faces involved in the proof of Claim \ref{(ii)} merge into a single $2$-dimensional face, which is given by the vertices $(d+2,0,0)$, $(d,2,0)$ and $(d,0,n+1)$. Then, by the same argument used in that proof, we get the opposite inequality:
\begin{equation*}
\deg \zeta_{\pi^* f_s\vert_{B_\varepsilon(P_{s,k})}}(t)\leq\deg \zeta_{\pi^* f_0\vert_{B_\varepsilon(P_{0,k})}}(t),
\end{equation*}
and we cannot conclude the proof of the theorem in the same way as above.
\end{remark}

\subsection{Proof of Theorem \ref{mt1} -- case (3)}

This case is quite simple. First, we observe that since $m=1$, each singular point of $C(f_{0,d})$ contributes to the second zeta-multiplicity factor $\zeta^{(2)}_{f_0,\mathbf{0}}(t)$ by the term $(1-t^{d+1})^{-1}$, so that 
\begin{equation}\label{zfdp22}
\zeta^{(2)}_{f_0,\mathbf{0}}(t)=(1-t^{d+1})^{-k_0}.
\end{equation} 
Secondly, let us pick a singular point $P_k\in\mbox{Sing}(C(f_{s,d}))$, $s\not=0$, and consider admissible coordinates $\mathbf{x}=(x_1,x_2,x_3)$ near $P_k$ in which $\pi^* f_s$ is given by 
\begin{equation}\label{exppbicx2}
\pi^* f_s=x_1^{d}\bigg(h(x_2,x_3)+\sum_{j\geq 1}x_1^j g_j(x_2,x_3)\bigg),
\end{equation}
where $h(x_2,x_3)$ is the defining polynomial of $C(f_{s,d})$ near $P_k$.
If \eqref{exppbicx2} includes a term of the form $a\,x_1$ with $a\not=0$, then $P_k$ contributes to the second zeta-multiplicity factor $\zeta^{(2)}_{\pi^* f_s\vert_{B_\varepsilon(P_k)}}(t)$ by the term $(1-t^{d+1})^{-1}$. On the other hand, if \eqref{exppbicx2} does not include such a term, then $\deg \pi^* f_s\geq d+2$, and the second zeta-multiplicity factor $\zeta^{(2)}_{\pi^* f_s\vert_{B_\varepsilon(P_k)}}(t)$ cannot contain the factor $(1-t^{d+1})^{-1}$. 

Thus, altogether, if there exists a point $P_k\in\mbox{Sing}(C(f_{s,d}))$ for which \eqref{exppbicx2} does not include the linear term $x_1$ with a non-zero coefficient, then, by Lemma \ref{ml3-0}, 
\begin{equation*}
\zeta^{(2)}_{f_s,\mathbf{0}}(t)\not=(1-t^{d+1})^{-k_0}\overset{\eqref{zfdp22}}{=}\zeta^{(2)}_{f_0,\mathbf{0}}(t),
\end{equation*}  
which is a contradiction (again we recall that, by \cite[Th\'eor\`eme 3.9]{Teissier2}, $\zeta_{f_0,\mathbf{0}}(t)=\zeta_{f_s,\mathbf{0}}(t)$). 
Therefore, for each $P_k\in\mbox{Sing}(C(f_{s,d}))$, the Newton principal part of $\tilde f$ must be of the form
\begin{equation*}
h(x_2,x_3)+a\, x_1,\ a\not=0,
\end{equation*}
 and the theorem follows from Lemma \ref{ml2} which says that the local embedded topological types of the singularities of $C(f_{s,d})$ are independent of $s$.

\appendix 
\section{Zeta-multi\-pli\-cities and zeta-multi\-pli\-city factors}\label{AppA1}

By the A'Campo--Oka formula (see \cite[Th\'eor\`eme 3]{AC} and \cite[Chapter I, Theorem (5.2)]{O1}), the monodromy zeta-function $\zeta_{f,\mathbf{0}}(t)$ of $f$ at $\mathbf{0}$ is uniquely written as
\begin{equation}\label{sectmszp-expzf}
\zeta_{f,\mathbf{0}}(t)=\prod_{i=1}^{\ell_0} (1-t^{d_i})^{\nu_i},
\end{equation}
where $d_1,\ldots,d_{\ell_0}$ are mutually disjoint and $\nu_1,\ldots,\nu_{\ell_0}$ are non-zero integers. 

\begin{definition}
The \emph{first zeta-multi\-pli\-city} of $\zeta_{f,\mathbf{0}}(t)$ is defined as the smallest integer $d_{i_0}$ among $d_1,\ldots,d_{\ell_0}$, and the corresponding factor $(1-t^{d_{i_0}})^{\nu_{i_0}}$ in \eqref{sectmszp-expzf} is called the \emph{first zeta-multiplicity factor} of $\zeta_{f,\mathbf{0}}(t)$ and is denoted by $\zeta^{(1)}_{f,\mathbf{0}}(t)$.

The \emph{second zeta-multi\-pli\-city} of $\zeta_{f,\mathbf{0}}(t)$ is the smallest $d_{i_1}$ in the set $\{d_1,\ldots,d_{\ell_0}\}\setminus \{d_{i_0}\}$, and the corresponding factor $(1-t^{d_{i_1}})^{\nu_{i_1}}$ in \eqref{sectmszp-expzf} is called the \emph{second zeta-multiplicity factor} of $\zeta_{f,\mathbf{0}}(t)$ and is denoted by $\zeta^{(2)}_{f,\mathbf{0}}(t)$.

Similarly, we define the $\ell$-th \emph{zeta-multi\-pli\-city} and the \emph{$\ell$th zeta-multiplicity factor} $\zeta^{(\ell)}_{f,\mathbf{0}}(t)$ for the other values of $\ell$, $3\leq \ell\leq \ell_0$.
\end{definition}

\bibliographystyle{amsplain}

\end{document}